\documentclass[11pt,english,reqno]{amsart}

\usepackage{amssymb,enumerate}

\usepackage[T1]{fontenc}
\usepackage[all]{xy}

\usepackage{tikz}
\usepackage{indentfirst} 
\usepackage{babel}
\usepackage{amstext}
\usepackage{amsmath}
\usepackage{amsfonts}

\usepackage{latexsym}
\usepackage{ifthen}
\usepackage{setspace}

\usepackage[colorlinks,linkcolor=red,anchorcolor=green,citecolor=blue]{hyperref}
\hypersetup{linktocpage = true}
\usepackage{enumitem}

\newcommand\cE{{\mathcal E}}
\newcommand\cF{{\mathcal F}}
\newcommand\cG{{\mathcal G}}

\newcommand\cK{{\mathcal K}}
\newcommand\cL{{\mathcal L}}

\newcommand\cO{{\mathcal O}}

\newcommand\cS{{\mathcal S}}

\newcommand\cV{{\mathcal V}}

\newcommand\bbC{{\mathbb C}}

\newcommand\bbP{{\mathbb P}}
\newcommand\bbQ{{\mathbb Q}}

\newcommand{\codim}{{\rm codim}}

\DeclareMathOperator*{\Sing}{Sing}

\DeclareMathOperator*{\supp}{Supp}

\newcommand{\Chow}[1]{\ensuremath{\mbox{\rm Chow}(#1)}}

\newcommand{\Hilb}[1]{\ensuremath{\mbox{\rm Hilb}(#1)}}
\newcommand{\Hom}[2]{\ensuremath{\mbox{$\rm{Hom}$}(#1,#2)}}
\newcommand{\Hombir}[2]{\ensuremath{\mbox{$\rm{Hom}_1$}(#1,#2)}}
\newcommand{\cHom}[2]{\ensuremath{\mathcal{H}om_{\mathcal{O}_X}(#1,#2)}}

\newcommand\cTor{{\mathcal Tor}}

\newcommand{\morp}[3]{\ensuremath{#1\colon #2\rightarrow #3}}

\newcommand{\RC}[1]{\ensuremath{\mbox{\rm RatCurves}^n(#1)}}

\newtheorem{lemma1}{}[section]
\newenvironment{lemma}{\begin{lemma1}{\bf Lemma.}}{\end{lemma1}}
\newenvironment{example}{\begin{lemma1}{\bf Example.}\rm}{\end{lemma1}}
\newenvironment{thm}{\begin{lemma1}{\bf Theorem.}}{\end{lemma1}}
\newenvironment{prop}{\begin{lemma1}{\bf Proposition.}}{\end{lemma1}}
\newenvironment{cor}{\begin{lemma1}{\bf Corollary.}}{\end{lemma1}}
\newenvironment{remark}{\begin{lemma1}{\bf Remark.}\rm}{\end{lemma1}}

\newenvironment{defn}{\begin{lemma1}{\bf Definition.}}{\end{lemma1}}

\newenvironment{convention}{{\bf Convention.}}{}
\newenvironment{thm A}{{\bf Theorem A.}}{}
\newenvironment{thm B}{{\bf Theorem B.}}{}
\newenvironment{thm C}{{\bf Theorem C.}}{}
\newenvironment{thm D}{{\bf Theorem D.}}{}
\newenvironment{remark*}{{\bf Remark.}}{}
\newenvironment{example*}{{\bf Example.}}{}
\newenvironment{assumption*}{{\bf Assumption.}}{}

\setlist[itemize]{leftmargin=*}
\setlist[enumerate]{leftmargin=*}

\newcommand{\bigslant}[2]{{\raisebox{.2em}{$#1$}\left/\raisebox{-.2em}{$#2$}\right.}}

\makeatletter

\title{CHARACTERIZATION OF PROJECTIVE SPACES AND $\mathbb{P}^r$-BUNDLES AS AMPLE DIVISORS}

\author{Jie Liu}

\address{Jie Liu, Universit{\'e} de C{\^o}te d'Azur, CNRS, LJAD, France}

\email{jliu@unice.fr}


\begin{document}
	
\begin{abstract}
	Let $X$ be a projective manifold of dimension $n$. Suppose that $T_X$ contains an ample subsheaf. We show that $X$ is isomorphic to $\bbP^n$. As an application, we derive the classification of projective manifolds containing a $\bbP^r$-bundle as an ample divisor by the recent work of D.~Litt.
\end{abstract}
	
\maketitle
	
\tableofcontents
	
\vspace{-0.2cm}

\section{Introduction}

Projective spaces are the simplest algebraic varieties. They can be characterized in many ways. A very famous one is given by the Hartshorne's conjecture, which was proved by Mori.

\begin{thm A}\cite[Theorem 8]{Mori1979}
	Let $X$ be a projective manifold defined over an algebraically closed field $k$ of characteristic $\geq 0$. Then $X$ is a projective space if and only if $T_X$ is ample.
\end{thm A}

This result has been generalized, over the field of complex number, by several authors \cite{Wahl1983,CampanaPeternell1998,AndreattaWisniewski2001}.

\begin{thm B}\cite[Theorem]{AndreattaWisniewski2001}
	Let $X$ be a projective manifold of dimension $n$. If $T_X$ contains an ample locally free subsheaf $\cE$ of rank $r$, then $X\cong\bbP^n$ and $\cE\cong\cO(1)^{\oplus r}$ or $\cE\cong T_{\bbP^n}$.
\end{thm B}

This theorem was successively proved for $r=1$ by Wahl \cite{Wahl1983} and latter for $r\geq n-2$ by Campana and Peternell \cite{CampanaPeternell1998}. The proof was finally completed by Andreatta and Wi{\'s}niewski \cite{AndreattaWisniewski2001}. The main aim of the present article is to prove the following generalization.

\begin{thm}\label{Main thm}
	Let $X$ be a projective manifold of dimension $n$. Suppose that $T_X$ contains an ample subsheaf $\cF$ of positive rank $r$, then $(X,\cF)$ is isomorphic to $(\bbP^n,T_{\bbP^n})$ or $(\bbP^n,\cO_{\bbP^	n}(1)^{\oplus r})$. 
\end{thm}

we refer to $\S$ \ref{ample sheaf} for the basic definition and properties of ample sheaves. Comparing with Theorem B, we don't require a priori the locally freeness of the subsheaf $\cF$ in Theorem \ref{Main thm}. In case the Picard number of $X$ is one, Theorem \ref{Main thm} is proved in \cite{AproduKebekusPeternell2008}. In fact, in \cite{AproduKebekusPeternell2008}, they showed that the subsheaf $\cF$ must be locally free under the additional assumption $\rho(X)=1$ and then Theorem B implies immediately Theorem \ref{Main thm}. In particular, to prove Theorem \ref{Main thm}, it suffices to show that $X$ is isomorphic to some projective space if its tangent bundle contains an ample subsheaf $\cF$, then the locally freeness of $\cF$ follows from \cite{AproduKebekusPeternell2008}. An interesting and important special case of Theorem \ref{Main thm} is when the subsheaf $\cF$ comes from the image of an ample vector bundle $E$ over $X$, this confirms a conjecture of Litt in \cite{Litt2016}.

\begin{cor}\label{Main Cor}
	Let $X$ be a projective manifold of dimension $n$, $E$ an ample vector bundle on $X$. If there exists a non-zero map $E\rightarrow T_X$, then $X\cong \bbP^n$.
\end{cor}

As an application, we derive the classification of projective manifolds containing a $\bbP^r$-bundle as an ample divisor. This problem has attracted a lot of interest over the past few decades (see \cite{Sommese1976,Badescu1984,FaniaSatoSommese1987,BeltramettiSommese1995,BeltramettiIonescu2009} etc.). Recently in \cite[Corollary 7]{Litt2016} Litt proved that it can be reduced to Corollary \ref{Main Cor}. To be more precise, we have the following classification theorem.

\begin{thm}\label{Conjecture BS}
	Let $X$ be a projective manifold of dimension $n\geq 3$, let $A$ be an ample divisor on $X$. Assume that $A$ is a $\bbP^r$-bundle, $\morp{p}{A}{B}$, over a manifold $B$ of dimension $>0$. Then one of the following holds:
	\begin{enumerate}
		\item $(X,A)=(\bbP(E),H)$ for some ample vector bundle $E$ over $B$ such that  $H\in\vert\cO_{\bbP(E)}(1)\vert$. $p$ is equal to the restriction to $A$ of the induced projection $\bbP(E)\rightarrow B$.
		
		\item $(X,A)=(\bbP(E),H)$ for some ample vector bundle $E$ over $\bbP^1$ such that $H\in\vert \cO_{\bbP(E)}(1)\vert$. $H=\bbP^1\times\bbP^{n-2}$ and $p$ is the projection to the second factor. 
		
		\item $(X,A)=(Q^3,H)$, where $Q^3$ is a smooth quadric threefold and $H$ is a smooth quadric surface with $H\in\vert \cO_{Q^3}(1)\vert$. $p$ is the projection to one of the factors $H\cong \bbP^1\times \bbP^1$.
		
		\item $(X,A)=(\bbP^3,H)$. $H$ is a smooth quadric surface and $H\in\vert\cO_{\bbP^3}(2)\vert$, and $p$ is again a projection to one of the factors of $H\cong\bbP^1\times\bbP^1$. 
	\end{enumerate}
\end{thm}

\begin{convention}
	Throughout we work over the field $\bbC$ of complex numbers unless otherwise stated. Varieties are always assumed to be integral separated schemes of finite type over $\bbC$. If $D$ is a Weil divisor on a projective normal variety $X$, we denote by $\cO_X(D)$ the reflexive sheaf associated to $D$. Given a coherent sheaf $\cF$ on a variety $X$ of generic rank $r$, then we denote by $\cF^{\vee}$ the sheaf $\cHom{\cF}{\cO_X}$ and by $\det(\cF)$ the sheaf $(\wedge^r\cF)^{\vee\vee}$. We denote by $\cF(x)=\cF_x\otimes_{\cO_{X,x}} k(x)$ the fiber of $\cF$ at $x\in X$.  If $\cF$ is a coherent sheaf on a variety $X$, we denote by $\bbP(\cF)$ the Grothendieck projectivization $\rm{Proj}\left(\oplus_{m\geq 0}\rm{Sym}^m \cF\right)$. If $\morp{f}{X}{Y}$ is a morphism between projective normal varieties, we denote by $\Omega^1_{X/Y}$ the relative differential sheaf. Moreover, if $Y$ is smooth, we denote by $K_{X/Y}$ the relative canonical divisor $K_X-f^*K_Y$ and by $\omega_{X/Y}$ the reflexive sheaf $\omega_X\otimes f^*\omega^\vee_Y$.
\end{convention}

\section{Ample sheaves and rational curves}\label{Section 2}

Let $X$ be a projective manifold. In this section, we gather some results about the behavior of an ample subsheaf $\cF\subset T_X$ with respect to a family of minimal rational curves on $X$. 

\subsection{Ample sheaves}\label{ample sheaf}
Recall that an invertible sheaf $\cL$ on a quasi-projective variety $X$ is said to be \emph{ample} if for every coherent sheaf $\cG$ on $X$, there is an integer $n_0>0$ such that for every $n\geq n_0$, the sheaf $\cG\otimes\cL^n$ is generated by its global sections \cite[\S~II.7]{Hartshorne1977}. In general, a coherent sheaf $\cF$ on a quasi-projective variety $X$ is said to be \emph{ample} if the invertible sheaf $\cO_{\bbP(\cF)}(1)$ is ample on $\bbP(\cF)$ \cite{Kubota1970}.

Well-known properties of ampleness of locally free sheaves still hold in this general setting.

\begin{enumerate}		
	\item A sheaf $\cF$ on a quasi-projective variety $X$ is ample if and only if, for any coherent sheaf $\cG$ on $X$, $\cG\otimes\rm{Sym}^m\cF$ is globally generated for $m\gg 1$ \cite[Theorem 1]{Kubota1970}.
	
	\item If $\morp{i}{Y}{X}$ is an immersion, and $\cF$ is an ample sheaf on $X$, then $i^*\cF$ is an ample sheaf on $Y$ \cite[Proposition 6]{Kubota1970}.
	
	\item If $\morp{\pi}{Y}{X}$ is a finite morphism with $X$ and $Y$ quasi-projective varieties, and $\cF$ is a coherent sheaf on $X$, then $\cF$ is ample if and only if $\pi^*\cF$ is ample. Note that $\bbP(\pi^*\cF)=\bbP(\cF)\times_X Y$ and $\cO_{\bbP(\cF)}(1)$ pulls back, by a finite morphism, to $\cO_{\bbP(\pi^*\cF)}(1)$.
	
	\item Any quotient of an ample sheaf is ample \cite[Proposition 1]{Kubota1970}. In particular, the image of an ample sheaf under a non-zero map is also ample.
	
	\item If $\cF$ is a locally free ample sheaf of rank $r$, then the $s^{th}$ exterior power $\wedge^s \cF$ is ample for any $1\leq s\leq r$ \cite[Corollary 5.3]{Hartshorne1966}.
	
	\item If $\cL$ is an ample invertible sheaf on a quasi-projective variety $X$, then $\cL^m$ is very ample for some $m>0$, i.e., there is an immersion $i\colon X\rightarrow\bbP^n$ for some $n$ such that $\cL^m=i^*\cO_{\bbP^n}(1)$ \cite[II, Theorem 7.6]{Hartshorne1977}.
\end{enumerate}

\subsection{Minimal rational curves}

Let $X$ be a normal projective variety. We will denote by $\Hom{\bbP^1}{X}$ the open subscheme $\subset \Hilb{\bbP^1\times X}$ of morphisms from $\bbP^1$ to $X$. Let $\Hombir{\bbP^1}{X}\subset\Hom{\bbP^1}{X}$ be the open subscheme corresponding to those morphisms $f\colon\bbP^1\rightarrow X$ which are birational onto their image. The group $Aut(\bbP^1)$ acts on $\Hombir{\bbP^1}{X}$ and its quotient "really parametrizes" morphisms from $\bbP^1$ into $X$. It can be proved that the quotient exists and its normalization will be denoted $\RC{X}$ and called the \emph{space of rational curve} on $X$. For more details we refer to \cite{Kollar1996}.

Let $\cV$ be an irreducible component of $\RC{X}$. $\cV$ is said to be \emph{a covering family of rational curves} on $X$ if the corresponding universal family dominates $X$. A covering family $\cV$ of rational curves on $X$ is called \emph{minimal} if its general members have minimal degree with respect to some polarization. If $X$ is a uniruled projective manifold, then $X$ carries a minimal covering family of rational curves. We fix such a family $\cV$, and let $[\ell]\in\cV$ be a general point. Then the tangent bundle $T_X$ can be decomposed on the normalization of $\ell$ as $\cO_{\bbP^1}(2)\oplus\cO_{\bbP^1}(1)^{\oplus d}\oplus\cO_{\bbP^1}^{\oplus (n-d-1)}$, where $d+2=\det(T_X)\cdot\ell\geq 2$ is the \emph{anticanonical degree} of $\cV$.

Let $\bar{\cV}$ be the normalization of the closure of $\cV$ in $\Chow{X}$. We define the following equivalence relation on $X$. Two points $x$, $y\in X$ are $\bar{\cV}$-equivalent if they can be connected by a chain of $1$-cycle from $\bar{\cV}$. By \cite{Campana1992} (see also \cite{KollarMiyaokaMori1992}), there exists a proper surjective morphism $\varphi_0\colon X_0\rightarrow T_0$ from an open subset of $X$ onto a normal variety $T_0$ whose fibers are $\bar{\cV}$-equivalence classes. We call this map the \emph{$\bar{\cV}$-rationally connected quotient} of $X$.

The first step towards Theorem \ref{Main thm} is the following result which was essentially proved in \cite{Araujo2006}.

\begin{thm}\label{Araujo Extension}\cite[Proposition 2.7]{AraujoDruelKovacs2008}
	Let $X$ be a projective uniruled manifold, $\cV$ a minimal covering family of rational curves on $X$. If $T_X$ contains a subsheaf $\cF$ of rank $r$ such that $\cF\vert_{\ell}$ is an ample vector bundle for a general member $[\ell]\in\cV$, then there exists a dense open subset $X_0$ of $X$ and a $\bbP^{d+1}$-bundle $\varphi_0\colon X_0\rightarrow T_0$ such that any curve on $X$ parametrized by $\cV$ and meeting $X_0$ is a line on a a fiber of $\varphi_0$. In particular, $\varphi_0$ is the $\bar{\cV}$-rationally connected quotient of $X$. 
\end{thm}

Recall that \emph{the singular locus} $\rm{Sing}(\cS)$ of a coherent sheaf $\cS$ over $X$ is the set of all points of $X$ where $\cS$ is not locally free.

\begin{remark}
	The hypothesis in Theorem \ref{Araujo Extension} that $\cF$ is locally free over a general member of $\cV$ is automatically satisfied. In fact, since $\cF$ is torsion free and $X$ is smooth, $\cF$ is locally free in codimension one. By \cite[II, Proposition 3.7]{Kollar1996}, a general member of $\cV$ is disjoint from $\rm{Sing}(\cF)$, hence $\cF$ is locally free over a general member of $\cV$.
\end{remark}

As an immediate application of Theorem \ref{Araujo Extension}, we can derive a weak version of \cite[Theorem 4.2]{AproduKebekusPeternell2008}.

\begin{cor}\label{AKP Cor}
	Let $X$ be a projective uniruled manifold with $\rho(X)=1$, $\cV$ a minimal covering family of rational curves on $X$. If $T_X$ contains a subsheaf $\cF$ of rank $r$ such that $\cF\vert_{\ell}$ is ample for a general member $[\ell]\in\cV$, then $X\cong \bbP^n$.
\end{cor}

\begin{cor}\cite[Corollary 4.3]{AproduKebekusPeternell2008}\label{AKP}
	Let $X$ be a projective manifold with $\rho(X)=1$. Assume that $T_X$ contains an ample subsheaf, then $X\cong\bbP^n$.
\end{cor}

\begin{proof}
	Since the tangent bundle $T_X$ contains an ample subsheaf $\cF$, $X$ is uniruled \cite[Corollary 8.6]{Miyaoka1987} and it carries a minimal covering family $\cV$ of rational curves . Note that the restriction $\cF\vert_C$ is ample for any curve $C\subset X$, thus we can deduce the result from Corollary \ref{AKP Cor}.
\end{proof}

\begin{remark}\label{Locally Freeness}
	Our approach above is quite different from that in \cite{AproduKebekusPeternell2008}. The proof in \cite{AproduKebekusPeternell2008} is based on a careful analysis of the singular locus of $\cF$ and the locally freeness of $\cF$ has been proved. Even though our argument does not tell anything about the singular locus of $\cF$, it has the advantage to give a rough description of the geometric structure of projective manifolds whose tangent bundle contains a "positive" subsheaf. 
\end{remark}

\section{Foliations and Pfaff fields}

Let $\cS$ be a subsheaf of $T_X$ on a quasi-projective manifold $X$. We denote by $\cS^{reg}$ the largest open subset of $X$ such that $\cS$ is a subbundle of $T_X$ over $\cS^{reg}$. Note that in general $\rm{Sing}(\cS)$ is a proper subset of $X\setminus\cS^{reg}$.

\begin{defn}
	Let $X$ be a quasi-projective manifold and let $\cS\subsetneq T_X$ be a coherent subsheaf of positive rank. $\cS$ is called a foliation if it satisfies the following conditions:
	\begin{enumerate}
		\item $\cS$ is saturated in $T_X$, i.e., $T_X/\cS$ is torsion free. 
		
		\item The sheaf $\cS$ is closed under the Lie bracket.		
	\end{enumerate}
	In addition, $\cS$ is called an algebraically integrable foliation if the following holds.
	\begin{enumerate}
		\item[(iii)] For a general point $x\in X$, there exists a projective subvariety $F_x$ passing through $x$ such that 
		\[\cS\vert_{F_x\cap \cS^{reg}}=T_{F_x}\vert_{F_x\cap\cS^{reg}}\subset T_X\vert_{F_x\cap \cS^{reg}}.\]
		We call $F_x$ the $\cS$-leaf through $x$.
	\end{enumerate}
\end{defn}

\begin{remark}\label{Foliation Generic}
	Let $X$ be a projective manifold and $\cS$ a saturated subsheaf of $T_X$. To show that $\cS$ is an algebraically integrable foliation, it is sufficient to show that it is an algebraically integrable foliation over a Zariski open subset of $X$.
\end{remark}

\begin{example}
	Let $X\rightarrow Y$ be a fibration with $X$ and $Y$ projective manifolds. Then $T_{X/Y}\subset T_X$ defines an algebraically integrable foliation on $X$ such that the general leaves are the fibers.
\end{example}

\begin{example}\cite[4.1]{AraujoDruel2013}\label{EX 2}
	Let $\cF$ be a subsheaf $\cO_{\bbP^n}(1)^{\oplus r}$ of $T_{\bbP^n}$ on $\bbP^n$. Then $\cF$ is an algebraically integrable foliation and it is defined by a linear projection $\bbP^n\dashrightarrow \bbP^{n-r}$. The set of points of indeterminacy $S$ of this rational map is a $r-1$-dimensional linear subspace. Let $x\not\in S$ be a point. Then the leaf passing through $x$ is the $r$-dimensional linear subspace $L$ of $\bbP^n$ containing both $x$ and $S$.
\end{example}

\begin{defn}
	Let $X$ be a projective variety, and $r$ a positive integer. A Pfaff field of rank $r$ on $X$ is a nonzero map $\partial\colon\Omega_X^r\rightarrow \cL$, where $\cL$ is an invertible sheaf on $X$.
\end{defn}

\begin{lemma}\cite[Proposition 4.5]{AraujoDruelKovacs2008}\label{Lift Pfaff}
	Let $X$ be a projective variety and $\morp{n}{\widetilde{X}}{X}$ its normalization. Let $\cL$ be an invertible sheaf on $X$, $r$ positive integer, and $\partial\colon \Omega_X^r\rightarrow\cL$ a Pfaff field. Then $\partial$ can be extended uniquely to a Pfaff field $\widetilde{\partial}\colon \Omega_{\widetilde{X}}^r\rightarrow n^*\cL$.
\end{lemma}

Let $X$ be a projective manifold and $\cS\subset T_X$ a subsheaf with positive rank $r$. We denote by $K_{\cS}$ the canonical class $-c_1(\det(\cS))$ of $\cS$. Then there is a natural associated Pfaff field of rank $r$:
\[\Omega_X^r=\wedge^r(\Omega_X^1)=\wedge^r(T_X^{\vee})=(\wedge^r T_X)^{\vee}\rightarrow \cO_X(K_{\cS}).\]

\begin{lemma}\cite[Lemma 3.2]{AraujoDruel2013}\label{Chow}
	Let $X$ be a projective manifold, and $\cS$ an algebraically integrable foliation on $X$. Then there is a unique irreducible projective subvariety $W$ of $\Chow{X}$ whose general point parametrizes a general leaf of $\cS$.
\end{lemma}

\begin{remark}\label{Lift of Pfaff}
	Let $X$ be a projective manifold, and $\cS$ an algebraically integrable foliation of rank $r$ on $X$. Let $W$ be the subvariety of \Chow{X} provided in Lemma \ref{Chow}. Let $Z\subset W$ be a general closed subvariety of $W$ and let $U\subset Z\times X$ be the universal cycle over $Z$. Let $\widetilde{Z}$ and $\widetilde{U}$ be the normalizations of $Z$ and $U$, respectively. We claim that the Pfaff field $\Omega_X^r\rightarrow \cO_X(K_{\cS})$ can be extended to a Pfaff field $\Omega_{\widetilde{U}/\widetilde{Z}}^r\rightarrow n^*p^*\cO_X(K_{\cS})$.
	\[
	\xymatrix{
		\widetilde{U}\ar[r]^{n}\ar[d]_{\widetilde{q}} & U\ar[d]^{q} \ar@{}[r]|-*[@]{\subset} & Z\times X \ar[d]^{q}\ar[r]^p& X\\
		\widetilde{Z}\ar[r]           & Z\ar[r]^{=} & Z
	}
	\]
	
	Let $V$ be the universal cycle over $W$ with $v\colon V\rightarrow X$. From the proof of \cite[Lemma 3.2]{AraujoDruel2013}, we know that the Pfaff field $\Omega_X^r\rightarrow\cO_X(K_{\cS})$ extends to be a Pfaff field $\Omega_{V}^r\rightarrow v^*\cO_X(K_{\cS})$. It induces a Pfaff field $\Omega_U^r\rightarrow p^*\cO_X(K_{\cS})$. Note that $U$ is irreducible since $Z$ is a general subvariety. By Lemma \ref{Lift Pfaff}, it can be uniquely extended to a Pfaff field $\Omega_{\widetilde{U}}^r\rightarrow n^*p^*\cO_X(K_{\cS})$.
	
	Let $\cK$ be the kernel of the morphism $\Omega_{\widetilde{U}}^r\twoheadrightarrow \Omega_{\widetilde{U}/\widetilde{Z}}^r$. Let $F$ be a general fiber of $\widetilde{q}$ such that its image under $p\circ n$ is a $\cS$-leaf and the morphism $p\circ n$ restricted on $F$ is finite and birational. Let $x\in F$ be a point such that $F$ is smooth at $x$ and $p\circ n$ is an isomorphism at a neighborhood of $x$. Then the composite map $\Omega_{\widetilde{U}}^r\vert_F\twoheadrightarrow \Omega_{\widetilde{U}/\widetilde{Z}}^r\vert_F\twoheadrightarrow\Omega_F^r$ implies that the composite map 
	\[\cK\rightarrow\Omega_{\widetilde{U}}^r\rightarrow n^*p^*\cO_X(K_{\cS})\]
	vanishes in a neighborhood of $x$, hence it vanishes generically over $\widetilde{U}$. Since the sheaf $n^*p^*\cO_X(K_{\cS})$ is torsion-free, it vanishes identically and finally yields a Pfaff field $\Omega_{\widetilde{U}/\widetilde{Z}}^r\rightarrow n^*p^*\cO_X(K_{\cS})$.
\end{remark}

Let $X$ be a projective manifold, and $\cS\subset T_X$ a subsheaf. We define its saturation $\overline{\cS}$ as the kernel of the natural surjection $T_X\twoheadrightarrow \bigslant{(T_X/\cS)}{(torsion)}$. Then $\overline{\cS}$ is obviously saturated.

\begin{thm}\label{Foliation}
	Let $X$ be a projective manifold. Assume that $T_X$ contains an ample subsheaf $\cF$ of rank $r<\dim(X)$. Then its saturation $\overline{\cF}$ defines an algebraically integrable foliation on $X$, and the $\overline{\cF}$-leaf passing through a general point is isomorphic to $\bbP^{r}$.
\end{thm}

\begin{proof}
	Let $\varphi_0\colon X_0 \rightarrow T_0$ be as the morphism provided in Theorem \ref{Araujo Extension}. Since $\cF$ is locally free in codimension one, we may assume that no fiber of $\varphi_0$ is completely contained in $\rm{Sing}(\cF)$.
	
	The first step is to show that $\cF\vert_{X_0}\subset T_{X_0/T_0}$. Since $\varphi_0\colon X_0\rightarrow T_0$ is smooth, we get a short exact sequence of vector bundles,
	\[0\rightarrow T_{X_0/T_0}\rightarrow T_X\vert_{X_0}\rightarrow \varphi_0^* T_{T_0}\rightarrow 0.\]
	The composite map $\cF\vert_{X_0}\rightarrow T_X\vert_{X_0}\rightarrow \varphi_0^* T_{T_0}$ vanishes on a Zariski open subset of every fiber. Since $\varphi_0^* T_{T_0}$ is torsion-free, it vanishes identically, and it follows $\cF\vert_{X_0}\subset T_{X_0/T_0}$.
	
	Next we show that, after shrinking $X_0$ and $T_0$ if necessary, $\cF$ is actually locally free over $X_0$. By generic flatness theorem \cite[Th\'eor\`eme 6.9.1]{Grothendieck1965}, after shrinking $T_0$, we can suppose that $\left(\bigslant{T_X}{\cF}\right)\vert_{X_0}$ is flat over $T_0$. Let $F\cong\bbP^{d+1}$ be an arbitrary fiber of $\varphi_0$. The following short exact sequence of sheaves
	\[0\rightarrow \cF\vert_{X_0}\rightarrow T_X\vert_{X_0}\rightarrow \left(\bigslant{T_X}{\cF}\right)\vert_{X_0}\rightarrow 0\]
	induces a long exact sequence of sheaves
	\[\cTor(\left(\bigslant{T_X}{\cF}\right)\vert_{X_0},\cO_F)\rightarrow \cF\vert_F\rightarrow T_X\vert_F\rightarrow \left(\bigslant{T_X}{\cF}\right)\vert_F\rightarrow 0.\]
	Since $\left(\bigslant{T_X}{\cF}\right)\vert_{X_0}$ is flat over $T_0$, it follows that $\cF\vert_F$ is a subsheaf of $T_X\vert_F$, in particular, $\cF\vert_F$ is torsion-free. Without loss of generality, we may assume that the restrictions of $\cF$ on all fibers of $\varphi_0$ are torsion-free. By Remark \ref{Locally Freeness}, the restrictions of $\cF$ on all fibers of $\varphi_0$ are locally free, it yields particularly that the dimension of the fibers of $\cF$ is constant on every fiber of $\varphi_0$ due to $\cF(x)=(\cF\vert_F)(x)$. Note that no fiber of $\varphi_0$ is contained in $\rm{Sing}(\cF)$, we conclude that the dimension of the fibers $\cF(x)$ of $\cF$ is constant over $X_0$. Hence $\cF$ is locally free over $X_0$.
	
	Now we claim that $\overline{\cF}$ actually defines an algebraically integrable foliation on $X_0$. Let $F\cong \bbP^{d+1}$ be an arbitrary fiber of $\varphi_0$. We know that $(F,\cF\vert_F)$ is isomorphic to $(\bbP^{d+1},T_{\bbP^{d+1}})$ or $(\bbP^{d+1},\cO_{\bbP^{d+1}}(1)^{\oplus r})$ (cf.~Theorem B), therefore $\cF$ defines an algebraically integrable foliation over $X_0$ (cf.~Example \ref{EX 2}). Note that we have $\cF\vert_{X_0}=\overline{\cF}\vert_{X_0}$ since $\cF\vert_{X_0}$ is saturated in $T_{X_0}$. Hence $\overline{\cF}$ also defines an algebraically integrable foliation over $X$ (cf.~Remark \ref{Foliation Generic}).
\end{proof}

\begin{remark}\label{Ampleness Codim one}
	Since $\cF$ is locally free on $X_0$, it follows that $\cO_X(-K_{\cF})\vert_{X_0}$ is isomorphic to $\wedge^r(\cF\vert_{X_0})$ and the invertible sheaf $\cO_X(-K_{\cF})$ is ample over $X_0$. Moreover, as $\cF$ is locally free in codimension one, there exists an open subset $X'\subset X$ containing $X_0$ such that $\codim(X\setminus X')\geq 2$ and  $\cO_X(-K_{\cF})$ is ample on $X'$. 
\end{remark}

\section{Proof of main theorem}

The aim of this section is to prove Theorem \ref{Main thm}. Let $X$ be a normal projective variety, and $X\rightarrow C$ a surjective morphism with connected fibers onto a smooth curve. Let $\Delta$ be an effective Weil divisor on $X$ such that $(X,\Delta)$ is log-canonical over the generic point of $C$. In \cite[Theorem 5.1]{AraujoDruel2013}, they proved that $-(K_{X/C}+\Delta)$ cannot be ample. In the next theorem, we give a variant of this result which is the key ingredient in our proof of Theorem \ref{Main thm}.

\begin{thm}\label{Relative Positivity}
	Let $X$ be a normal projective variety, and $f\colon X\rightarrow C$ a surjective morphism with connected fibers onto a smooth curve. Let $\Delta$ be a Weil divisor on $X$ such that $K_X+\Delta$ is Cartier and $\Delta^{hor}$ is reduced. Assume that there exists an open subset $C_0$ such that the pair $(X,\Delta^{hor})$ is snc over $X_0=f^{-1}(C_0)$. If $X'\subset X$ is an open subset such that no fiber of $f$ is completely contained in $X\setminus X'$ and $X_0\subset X'$, then the invertible sheaf $\cO_X(-K_{X/C}-\Delta)$ is not ample over $X'$.
\end{thm}

\begin{proof}
	To prove the theorem, we assume to the contrary that the invertible sheaf $\cO_X(-K_{X/C}-\Delta)$ is ample over $X'$. Let $A$ be an ample divisor supported on $C_0$. Then for some $m\gg 1$, the sheaf $\cO_X(-m(K_{X/C}+\Delta)-f^*A)$ is very ample over $X'$ \cite[Corollaire 4.5.11]{Grothendieck1961}. It follows that there exists a prime divisor $D'$ on $X'$ such that the pair $(X',\Delta^{hor}\vert_{X'}+D')$ is snc over $X_0$ and 
	\[D'\sim (-m(K_{X/C}+\Delta)-f^*A)\vert_{X'}.\]
	It implies that there exists a rational function $h\in K(X')=K(X)$ such that the restriction of the Cartier divisor $D=div(h)-m(K_{X/C}+\Delta)-f^*A$ on $X'$ is $D'$ and $D^{hor}$ is the closure of $D'$ in $X$. Note that  we can write $D=D_+-D_-$ for some effective divisors $D_+$ and $D_-$ with no common components. Then we have $\supp(D_-)\subset X\setminus X'$, in particular, no fiber of $f$ is supported on $D_-$. By \cite[Theorem 4.15]{Kollar2013}, there exists a log-resolution $\mu\colon \widetilde{X}\rightarrow X$ such that:
	\begin{enumerate}
		\item The induced morphism $\widetilde{f}=f\circ\mu\colon \widetilde{X}\rightarrow C$ is prepared (cf.~\cite[\S~4.3]{Campana2004}).
		
		\item The birational morphism $\mu$ is an isomorphism over $X_0$.
		
		\item $\mu_*^{-1}\Delta^{hor}+\mu_*^{-1}D^{hor}$ is a snc divisor.
	\end{enumerate}	
	Let $E$ be the exceptional divisor of $\mu$. Note that we have $\widetilde{f}_*(E)\not=C$. Moreover, we also have
	\[K_{\widetilde{X}}+\mu_*^{-1}\Delta+\frac{1}{m}\mu_*^{-1}D_+=\mu^*(K_X+\Delta+\frac{1}{m}D)+\frac{1}{m}\mu_*^{-1}D_-+E_+-E_-.\] 
	where $E_+$ and $E_-$ are effective $\mu$-exceptional divisors with no common components. 
	
	Set $\widetilde{D}=m\mu^{-1}_*\Delta+\mu^{-1}_*D_+ +mE_-$. Then $\widetilde{D}^{hor}=m\mu^{-1}_*\Delta^{hor}+\mu^{-1}_* D^{hor}$ is a snc effective divisor with coefficients $\leq m$. Since $D$ is linearly equivalent to $-m(K_{X/C}+\Delta)-f^*A$, we can write
	\[K_{\widetilde{X}/C}+\frac{1}{m}\widetilde{D}\sim_{\bbQ}-\frac{1}{m}\widetilde{f}^*A+\frac{1}{m}\mu_*^{-1}D_-+E_+.\]
	After multiplying by some positive $l$ divisible enough, we may assume that $lm E_+$ and $lmE_-$ are of integer coefficients. By replacing $\widetilde{D}$ by $l\widetilde{D}$, the weak positivity theorem \cite[Theorem 4.13]{Campana2004} implies that the following direct image sheaf
	\begin{align*}
	\widetilde{f}_*(\omega^{lm}_{\widetilde{X}/C}\otimes \cO_{\widetilde{X}}(\widetilde{D})) & \simeq\widetilde{f}_*(\cO_{\widetilde{X}}(-l\widetilde{f}^*A+lm E_++l\mu_*^{-1}D_-))\\
	& \simeq\cO_C(-lA)\otimes \widetilde{f}_*\cO_{\widetilde{X}}(lmE_++l\mu_*^{-1}D_-)
	\end{align*}
	is weakly positive.
	
	Observe that $\widetilde{f}_*(\cO_{\widetilde{X}}(lmE_++l\mu_*^{-1}D_-))=\cO_C$. Indeed, $E_+$ is a $\mu$-exceptional divisor, it follows $\mu_*(\cO_{\widetilde{X}}(lmE_++l\mu_*^{-1}D_-))=\cO_X(lD_-)$. Note that we have $f_*(\cO_X(lD_-))=\cO_C(P)$ for some effective divisor $P$ on $C$ such that $\supp(P)\subset f(\supp(D_-))$. Let $V$ be an open subset of $C$ and let $\lambda\in H^0(V,\cO_C(P))$, that is, $\lambda$ is a rational function on $C$ such that $div(\lambda)+P\geq 0$ over $V$. It follows that $div(\lambda\circ f)+lD_-\geq 0$ over $f^{-1}(V)$. Since there is no fiber of $f$ completely supported on $D_-$, the rational function $\lambda\circ f$ is regular over $f^{-1}(V)$. Consequently, the rational function $\lambda$ is regular over $V$. It implies that the natural inclusion $\cO_C\rightarrow \cO_{C}(P)$ is surjective, which yields $\widetilde{f}_*(\cO_{\widetilde{X}}(lmE_+ +l\mu_*^{-1}D_-))=\cO_C$. However, this shows that $\cO_C(-lA)$ is weakly positive, a contradiction. Hence $\cO_{X}(-K_{X/C}-\Delta)$ is not ample over $X'$.
\end{proof}

\begin{lemma}\label{Relative divisor}
	Let $X$ be a normal projective variety, and $\morp{f}{X}{C}$ a surjective morphism with reduced and connected fibers onto a smooth curve $C$. Let $D$ be a Cartier divisor on $X$. If there exists a nonzero morphism $\Omega_{X/C}^r\rightarrow \cO_X(D)$, where $r$ is the relative dimension of $f$, then there exists an effective Weil divisor $\Delta$ on $X$ such that $K_{X/C}+\Delta=D$.
\end{lemma}

\begin{proof}
	Since all the fibers of $f$ are reduced, the sheaf $\Omega_{X/C}^r$ is locally free in codimension one. Hence the reflexive hull of $\Omega^r_{X/C}$ is $\omega_{X/C}\simeq\cO_X(K_{X/C})$. Note that  $\cO_X(D)$ is reflexive, the nonzero morphism $\Omega_{X/C}^r\rightarrow \cO_X(D)$ induces a nonzero morphism $\omega_{X/C}\rightarrow\cO_X(D)$. This shows that there exists an effective divisor $\Delta$ on $X$ such that $K_{X/C}+\Delta=D$.
\end{proof}

As an application of Theorem \ref{Relative Positivity}, we derive a special property about foliations defined by an ample subsheaf of $T_X$. A similar result was established for Fano foliations with mild singularities in the work of Araujo and Druel \cite[Proposition 5.3]{AraujoDruel2013} and we follow the same strategy. 

\begin{prop}\label{Common Point}
	Let $X$ be a projective manifold. If $\cF\subset T_X$ is an ample subsheaf of rank $r<n=\dim(X)$, then there is a common point in the closure of general leaves of $\overline{\cF}$.
\end{prop}

\begin{proof}
	Since $\cF$ is torsion-free and $X$ is smooth, $\cF$ is locally free over an open subset $X'\subset X$ such that $\codim(X\setminus X')\geq 2$, in particular, $\cO_X(-K_{\cF})$ is ample over $X'$. By Theorem \ref{Araujo Extension}, there exists an open subset $X_0\subset X$ and a $\bbP^{d+1}$-bundle $\varphi_0\colon X_0\rightarrow T_0$. Moreover, from the proof of Theorem \ref{Foliation}, the saturation $\overline{\cF}$ defines an algebraically integrable foliation on $X$ and we may assume that $\cF$ is locally free over $X_0$. In particular, we have $X_0\subset X'$. In view of Lemma \ref{Chow}, we will denote by $W$ the subvariety of $\Chow X$ parametrizing the general leaves of $\overline{\cF}$ and $V$ the normalization of the universal cycle over $W$. Let $\morp{p}{V}{X}$ and $\morp{\pi}{V}{W}$ be the natural projections. Note that there exists an open subset $W_0$ of $W$ such that $p(\pi^{-1}(W_0))\subset X_0$.
	
	To prove our proposition, we assume to the contrary that there is no common point in the general leaves of $\overline{\cF}$. 
	
	First we show that there exists a smooth curve $C$ with a finite morphism $n\colon C\rightarrow n(C)\subset W$ such that:
	\begin{enumerate}
		\item Let $U$ be the normalization of the fiber product $V\times_W C$ with projection $\morp{\pi}{U}{C}$. Then the induced morphism $\morp{\widetilde{p}}{U}{X}$ is finite onto its image.
		
		\item There exists an open subset $C_0$ of $C$ such that the image of $U_0$ under $p$ is contained in $X_0$. In particular, $U_0=\pi^{-1}(C_0)$ is a $\bbP^{r}$-bundle over $C_0$.
		
		\item For any point $c\in C$, the image of the fiber $\pi^{-1}(c)$ under $\widetilde{p}$ is not contained in $X\setminus X'$.
		
		\item All the fibers of $\pi$ are reduced.
	\end{enumerate}
	Note that we have $X\setminus X'=\Sing{(\cF)}$ and $\codim(\Sing{(\cF)})\geq 2$. We consider the subset 
	\[Z=\{w\in W\ \vert\ \pi^{-1}(w)\subset p^{-1}(\Sing(\cF))\}.\]
	Since $\pi$ is equidimensional, it is a surjective universally open morphism \cite[Th\'eor\`eme 14.4.4]{Grothendieck1966}. Therefore the subset $Z$ is closed. Note that the general fiber of $\pi$ is disjoint from $p^{-1}(\Sing(\cF))$, so $\codim(Z)\geq 1$. Moreover, by the definition of $Z$, we have $p(\pi^{-1}(Z))\subset \Sing(\cF)$ and $\codim(\Sing(\cF))\geq 2$, hence we can choose some very ample divisors $H_i$ $(1\leq i\leq n)$ on $X$ such that the curve $B$ defined by complete intersection $\widetilde{p}^*H_1\cap\cdots\cap\widetilde{p}^*H_n$
	satisfies the following conditions : 
	\begin{enumerate}[label=(\roman*$'$)]
		\item There is no common point in the closure of the general fibers of $\pi$ over $\pi(B)$.
		
		\item $\pi(B)\cap W_0\not=\emptyset$.
		
		\item $\pi(B)\subset W\setminus Z$.
	\end{enumerate}

	Let $B'\rightarrow B$ be the normalization, and $V_{B'}$ the normalization of the fiber product $V\times_B B'$. The induced morphism $V_{B'}\rightarrow V$ is denoted by $\mu$. Then it is easy to check that $B'$ satisfies (i), (ii) and (iii). By \cite[Theorem 2.1]{BoschLuetkebohmertRaynaud1995}, there exists a finite morphism $C\rightarrow B'$ such that all the fibers of $U\rightarrow C$ are reduced, where $U$ is the normalization of $V_{B'}\times_{B'} C$. Then we see at once that $C$ is the desired curve.
	
	The next step is to get a contradiction by applying Theorem \ref{Relative Positivity}. From Remark \ref{Lift of Pfaff}, we see that the Pfaff field $\Omega_X^r\rightarrow \cO_X(K_{\overline{\cF}})$ extends to a Pfaff field $\Omega_{V_{B'}/B'}^r\rightarrow \mu^*p^*\cO_X(K_{\overline{\cF}})$, and it induces a Pfaff field $\Omega_{U/C}^r\rightarrow \widetilde{p}^*\cO_X(K_{\overline{\cF}})$. The natural inclusion $\cF\hookrightarrow\overline{\cF}$ induces a morphism $\cO_X(K_{\overline{\cF}})\rightarrow\cO_X(K_{\cF})$. It implies that we have a Pfaff field $\Omega_{U/C}^r\rightarrow \widetilde{p}^*\cO_X(K_{\cF})$. By Lemma \ref{Relative divisor}, there exists an effective Weil divisor $\Delta$ on $U$ such that $K_{U/C}+\Delta=\widetilde{p}^*K_{\cF}$. 
	
	Let $\Delta^{hor}$ be the $\pi$-horizontal part of $\Delta$. After shrinking $C_0$, we may assume that $\Delta\vert_{U_0}=\Delta^{hor}\vert_{U_0}$. According to the proof of Theorem \ref{Foliation}, for any fiber $F\cong\bbP^r$ over $C_0$, we have $(\widetilde{p}^*{K_{\cF}})\vert_F-K_F=0$ or $H$ where $H\in\vert\cO_{\bbP^r}(1)\vert$. This shows that either $\Delta^{hor}$ is zero or $\Delta^{hor}$ is a prime divisor such that $\Delta\vert_{U_0}=\Delta^{hor}\vert_{U_0}\in\vert\cO_{U_0}(1)\vert$. In particular, the pair $(U,\Delta^{hor})$ is snc over $U_0$ and $\Delta^{hor}$ is reduced. Note that $\widetilde{p}\colon U\rightarrow \widetilde{p}(U)$ is a finite morphism, so the invertible sheaf $\widetilde{p}^*\cO_X(-K_{\cF})$ is ample over $U'=U\cap \widetilde{p}^{-1}(X')$, i.e., the sheaf $\cO_{U}(-K_{U/C}-\Delta)$ is ample over $U'$, which contradicts to Theorem \ref{Relative Positivity}.
\end{proof} 

Now our main result immediately follows:

\begin{proof}[Proof of Theorem \ref{Main thm}]
	Theorem \ref{Araujo Extension} implies that there exists an open subset $X_0\subset X$ and a normal variety $T_0$ such that $X_0\rightarrow T_0$ is a $\bbP^{d+1}$-bundle and $d+1\geq r$. Without loss of generality, we may assume $r<dim(X)$. By Theorem \ref{Foliation} followed by Proposition \ref{Common Point}, $\overline{\cF}$ defines an algebraically integrable foliation over $X$ such that there is a common point in the closure of general leaves of $\overline{\cF}$. However, this cannot happen if $\dim(T_0)\geq 1$. Hence we have $\dim T_0=0$ and $X\cong \bbP^n$. 
\end{proof}

\section{$\bbP^r$-bundles as ample divisors}

As an application of Theorem \ref{Main thm}, we classify projective manifolds $X$ containing a $\bbP^r$-bundle as an ample divisor. This is originally conjectured by Beltrametti and Sommese \cite[Conjecture 5.5.1]{BeltramettiSommese1995}. In the sequel of this section, we follow the same notation and assumptions as in Theorem \ref{Conjecture BS}. 

The case $r\geq 2$ follows from Sommese's extension theorem \cite[Proposition III]{Sommese1976} (see also \cite[Theorem 5.5.2]{BeltramettiSommese1995}). For $r=1$ and $b=1$, it is due to B{\u{a}}descu \cite[Theorem D]{Badescu1984} (see also \cite[Theorem 5.5.3]{BeltramettiSommese1995}). For $r=1$ and $b=2$, it was done by the work of several authors \cite[Theorem 7.4]{BeltramettiIonescu2009}. As mentioned in introduction, Litt proved the following result by which we can deduce Theorem \ref{Conjecture BS} from Corollary \ref{Main Cor}.

\begin{prop}\cite[Lemma 4]{Litt2016}\label{Extension Litt}
	Let $X$ be a projective manifold of dimension $\geq 3$, and let $A$ be an ample divisor. Assume that $\morp{p}{A}{B}$ is a $\bbP^1$-bundle, then either $p$ extends to a morphism $\morp{\widehat{p}}{X}{B}$, or there exists an ample vector bundle $E$ on $B$ and a non-zero map $E\rightarrow T_B$.
\end{prop}

For the reader's convenience, we outline the argument of Litt that reduces Theorem \ref{Conjecture BS} to Corollary \ref{Main Cor}.

\begin{proof}[Proof of Theorem \ref{Conjecture BS}]
	Since the case $r\geq 2$ is already known, we can assume that $r=1$, i.e., \ $\morp{p}{A}{B}$ is a $\bbP^1$-bundle.
	
	If $p$ extends to a morphism $\morp{\widehat{p}}{X}{B}$, then the result follows from \cite[Theorem 5.5]{BeltramettiIonescu2009} and we are in the case (i) of theorem theorem.
	
	If $p$ does not extend to a morphism $X\rightarrow B$, by Proposition \ref{Extension Litt}, there exists an ample vector bundle $E$ over $B$ with a non-zero map $E\rightarrow T_B$. Own to Corollary \ref{Main Cor}, we have $B\cong\bbP^b$. As the case $b\leq 2$ is also known, we may assume that $b\geq 3$. In this case, by \cite[Theorem 2.1]{FaniaSatoSommese1987}, we conclude that $X$ is a $\bbP^{n-1}$-bundle over $\bbP^1$ and we are in the case (ii) of the theorem.
\end{proof}

\subsection*{Acknowledgements}
	I would like to express my deep gratitude to my advisor A.~H{\"o}ring for suggesting me to work on this question and also for his many valuable discussions, guidance and help during the preparation of this paper. I would like to express my sincere thanks to S.~Druel for his careful reading of the first draft of this paper and for kindly pointing out several mistakes. A special thanks is owed to M.~Beltrametti, D.~Litt and C.~Mourougane for their interest and helpful comments.


\def\cprime{$'$}

\renewcommand\refname{References}
\bibliographystyle{alpha}
\bibliography{MinimalRationalCurves}

\end{document}